\documentclass[11pt]{amsart}

\usepackage{amsmath,amsfonts,amssymb}
\usepackage{xypic}
\usepackage[english]{babel}
\usepackage{fullpage}
\usepackage{url}

\theoremstyle{theorem}
\newtheorem{thm}{Theorem}[section]
\newtheorem{prop}[thm]{Proposition}
\theoremstyle{definition}
\newtheorem{defi}[thm]{Definition}
\newtheorem{rem}[thm]{Remark}
\newtheorem{ex}[thm]{Example}

\newcommand{\Z}{\mathbb{Z}}
\newcommand{\Q}{\mathbb{Q}}
\newcommand{\R}{\mathbb{R}}
\newcommand{\C}{\mathbb{C}}
\newcommand{\T}{\mathbb{T}}

\newcommand{\definizione}{\;\stackrel{\text{def}}{=}\;}
\newcommand{\GL}{\mathrm{GL}}
\newcommand{\phit}[1]{\varphi^{#1}_{\mathbf{t}}}
\newcommand{\bphit}[1]{\bar{\varphi}^{#1}_{\mathbf{t}}}
\newcommand{\pphit}[2]{\phit{#1}\wedge\phit{#2}}

\newcommand{\pbphit}[2]{\phit{#1}\wedge\bphit{#2}}

\newcommand{\modulo}[1]{\left|#1\right|}
\newcommand{\correnti}{\mathcal{D}}
\newcommand{\trasposta}[1] {{{#1}^\mathrm{t}}}

\newcommand{\sspace}{\cdot}
\newcommand{\ssspace}{\cdot\cdot}

\DeclareMathOperator{\im}{i}
\DeclareMathOperator{\Span}{span}
\DeclareMathOperator{\de}{d}
\DeclareMathOperator{\id}{id}

\newcommand{\Cpf}{$\mathcal{C}^\infty$-pure-and-full}

\newcommand{\Cf}{$\mathcal{C}^\infty$-full}
\newcommand{\Cp}{$\mathcal{C}^\infty$-pure}
\newcommand{\pf}{pure-and-full}

\newcommand{\f}{full}
\newcommand{\p}{pure}
\newcommand{\del}{\partial}
\newcommand{\delbar}{\overline{\del}}
\title[On cohomological decomposition and deformations]{On cohomological
decomposition of almost-complex manifolds and deformations}
\author{Daniele Angella}
\address[Daniele Angella]{Dipartimento di Matematica "L. Tonelli"\\
Universit\`{a} di Pisa \\
Largo Bruno Pontecorvo 5, 56127\\
Pisa, Italy}
\email{angella@mail.dm.unipi.it}
\author{Adriano Tomassini}
\address[Adriano Tomassini]{Dipartimento di Matematica\\
Universit\`{a} di Parma \\
Parco Area delle Scienze 53/A, 43124 \\
Parma, Italy}
\email{adriano.tomassini@unipr.it}
\thanks{This work was supported by the Project MIUR
``Geometric Properties of Real and Complex Manifolds'' and by GNSAGA
of INdAM}

\begin{document}

\vspace{-2cm}
\begin{minipage}[l]{10cm}
{\sffamily
  D. Angella, A. Tomassini, On cohomological decomposition of almost-complex manifolds and deformations,
  {\em J. Symplectic Geom.} \textbf{9} (2011), no.~3, 403--428.

\medskip

  \begin{flushright}\begin{footnotesize}
  (First published in {\em J. Symplectic Geom.} in volume \textbf{9}, issue 3, 2011, published by International Press, \url{http://www.intlpress.com/JSG/}.)
  \end{footnotesize}\end{flushright}
}
\end{minipage}
\vspace{2cm}

\begin{abstract}
While small deformations of compact K\"{a}hler manifolds are K\"{a}hler too, we prove
that the cohomological property to be \Cpf\ is not a stable condition under small deformations.
This property, which has been recently introduced and studied by T.-J.
Li and W. Zhang in \cite{li-zhang} and T. Dr\v{a}ghici, T.-J. Li and W. Zhang
in \cite{draghici-li-zhang, draghici-li-zhang1},
is weaker than the K\"{a}hler one and characterizes the
almost-complex structures inducing a decomposition in cohomology.
We also study the stability of this property along curves of almost-complex structures
constructed starting from the harmonic representatives in special cohomology classes.
\end{abstract}

\maketitle

\section{Introduction}\label{sec:introduction}
Let $\left(M,J\right)$ be a compact almost-complex $2n$-dimensional manifold and let $\omega$ be a
symplectic form on $M$. Then $J$ is said to be \emph{$\omega$-tamed} if $\omega\left(\sspace,J\sspace\right) \,>\, 0$
and \emph{$\omega$-compatible} (or \emph{$\omega$-calibrated}) if $g\left(\cdot,\cdot\right) \,:=\,
\omega\left(\cdot,J\cdot\right)$ is a $J$-Hermitian metric. Define the \emph{tamed cone} $\mathcal{K}_J^t$
as the open convex cone given by the projection in cohomology of the space of the symplectic forms taming $J$, namely
$$
\mathcal{K}^t_J \;\definizione\; \left\{\left[\omega\right]\in H^2_{dR}(M;\R) \,\,\,\vert\,\,\, J\text{ is }\omega\text{-tamed}\right\} \;,
$$
and the \emph{compatible cone} $\mathcal{K}_J^c$ as its subcone given by the projection of the space of the symplectic
forms compatible with $J$, namely
$$
\mathcal{K}^c_J \;\definizione\; \left\{\left[\omega\right]\in H^2_{dR}(M;\R) \,\,\,\vert\,\,\, J\text{ is }
\omega\text{-compatible}\right\} \;.
$$

T.-J. Li and W. Zhang proved in \cite[Corollay 3.2]{li-zhang} that if $J$ is integrable and $\mathcal{K}^c_J$
is non-empty then the following relation between the two cones holds:
\begin{equation}\label{eq:cones}
\mathcal{K}^t_J \;=\; \mathcal{K}^c_J \;+\; \left(\left( H^{2,0}_{\delbar}\left(M\right)
\,\oplus\, H^{0,2}_{\delbar}\left(M\right) \right) \,\cap\, H^2_{dR}\left(M;\R\right) \right) \;;
\end{equation}
they also proved (see \cite[Theorem 1.2]{li-zhang}) that, given a
complex compact surface $(M,J)$, if there is a symplectic structure $\omega$
such that $J$ is $\omega$-tamed then $(M,J)$ admits a K\"ahler structure
(see also \cite[Proposition 1.6]{ST}), i.e. in such a case
$\mathcal{K}^t_J$ is empty if and only if $\mathcal{K}^c_J$
is empty: this gives a partial answer to a question of Donaldson's, \cite[Question 2]{donaldson}.
Therefore, the problem of finding explicit examples of compact complex non-K\"ahler manifolds, admitting a holomorphic
structure tamed by a symplectic form, makes sense only in dimension higher
than $4$, as asked by T.-J. Li and W. Zhang in \cite[p. 678]{li-zhang} and by J. Streets and G. Tian \cite[Question 1.7]{ST}.

In view of \cite[Theorem A]{benson-gordon1}, the most natural category in which
one can find non-K\"ahler manifolds is that one of nilmanifolds: we
prove that no such example could be found among the nilmanifolds of
dimension $6$ (see Theorem \ref{nilmanifold-no-tamed}).

In order to generalize \eqref{eq:cones} for an arbitrary almost-complex structure, T.-J. Li and W. Zhang introduced
in \cite{li-zhang} the concept of \Cpf\ almost-complex structure. More precisely, an almost-complex structure $J$ is
said to be \emph{\Cpf} if it induces the decomposition
$$
H^2_{dR}\left(M;\R\right) \;=\; H^{(1,1)}_J\left(M\right)_\R \;\oplus\; H^{(2,0),(0,2)}_J\left(M\right)_\R \;,
$$
where the group $H^{(2,0),(0,2)}_J(M)_\R$ (respectively, $H^{(1,1)}_J(M)_\R$) is given by the projection in cohomology
of the space $\left(\wedge^{2,0}M\oplus\wedge^{0,2}M\right)\,\cap\, \wedge^2M$ (respectively,
$\wedge^{1,1}M\,\cap\,\wedge^2M$); more in general, $J$ is said to be \emph{\Cf} if the equality
$$
H^2_{dR}\left(M;\R\right) \;=\; H^{(1,1)}_J\left(M\right)_\R \;+\; H^{(2,0),(0,2)}_J\left(M\right)_\R
$$
holds, namely if there exists a basis of $H^2_{dR}(M;\R)$ formed
by classes having at least one type of pure degree representative.\newline
%By definition, an almost-complex structure is said to be {\em closed}, if
%$\de\left[\wedge^{1,1}M\right]$ is closed in $\wedge^3(M;\R)$. For example, if $J$ is integrable, then $J$ is closed
%see \cite{li-zhang}.\newline
In \cite[Theorem 1.1]{li-zhang}, T.-J. Li and W. Zhang proved that if $J$ is \Cf\ and
if $\mathcal{K}^c_J$ is non-empty, then
$$
\mathcal{K}^t_J \;=\; \mathcal{K}^c_J \,+\, H^{(2,0),(0,2)}_J(M)_\R \;,
$$
where $H^{(2,0),(0,2)}_J(M)_\R$ generalizes the group
$$
\left(H^{2,0}_{\delbar}(M)
\oplus H^{0,2}_{\delbar}(M)\right)\,\cap\, H^2_{dR}(M;\R)
$$
in \eqref{eq:cones}.

In \cite{li-zhang} dual notions starting from the space of currents are also defined: we will recall in
Section \ref{sec:Cpf-structures} what a \emph{\pf} almost-complex structure is. Further studies about
\Cpf\ almost-complex structures have been carried out in \cite{draghici-li-zhang} and \cite{fino-tomassini}.\newline
In particular, T. Dr\v{a}ghici, T.-J. Li and W. Zhang proved in
\cite[Theorem 2.3]{draghici-li-zhang} that every almost-complex
structure on a compact $4$-dimensional manifold is \Cpf. \newline
As a consequence of the last two quoted results, (see \cite[Corollary 1.1]{li-zhang}), if $(M, J)$ is a compact almost complex
$4$-manifold such that $\mathcal{K}^c_J$ is non-empty, then
$$
\mathcal{K}^t_J \;=\; \mathcal{K}^c_J \,+\, H^{(2,0),(0,2)}_J(M)_\R \,.
$$
In particular, if $b^+(M)=\dim_\R H^{(1,1)}_J(M)_\R = 1,$ then $\mathcal{K}^t_J=\mathcal{K}^c_J$.

In real dimension greater than $4$, things are different. Indeed,
for example, there are almost-complex structures on
compact $6$-dimensional solvmanifolds which are not ${\mathcal C}^\infty$-pure (see
\cite[Example 3.3]{fino-tomassini}). This turns our attention to the $6$-dimensional case.

In this paper, we are interested in studying small deformations of \Cpf\ complex structures.
The celebrated theorem of K. Kodaira and D. C. Spencer, \cite[Theorem 15]{kodaira-spencer-3},
states that K\"{a}hler metrics on compact complex manifolds are stable under small deformations; L. Alessandrini and G. Bassanelli
proved in \cite{alessandrini-bassanelli} that this stability fails to be true for the class of $p$-{\em K\"{a}hler
manifolds}, where $p\in\left\{2,\ldots,n-1\right\}$ (see \cite{alessandrini-andreatta} for the precise definition),
e.g. for the class of \emph{balanced} metrics, namely the $J$-Hermitian metrics on compact complex manifolds
whose fundamental form $\omega$
satisfies $\de\omega^{n-1}=0$. \newline
Since the \Cpf\ condition is weaker than the K\"{a}hler one (more precisely, as a consequence of \cite{li-zhang},
see Theorem \ref{thm:frolicher},
every
compact complex manifold verifying the $\del\delbar$-Lemma is \Cpf), it could
be interesting to establish if the \Cpf\ complex structures are stable under small deformations.
As hinted by T.-J. Li and W. Zhang in a previous version of \cite{li-zhang}, we study the stability of
the standard complex structure on the Iwasawa manifold and try to deform a \Cpf\ almost-complex structure
starting with $J$-anti-invariant forms as explained in \cite{lee}.

In \cite{nakamura} I. Nakamura computed the small deformations of the Iwasawa manifold $X$, dividing them
in three
classes. Then, a direct computation shows that the complex structure on $X$ is \Cpf. \newline
We prove that
(see Theorem \ref{thm:instability-iwasawa} for the precise statement) the small deformations of class (i)
are \Cpf\ while those ones of classes (ii) and (iii) are not. Hence, as a corollary we get the
following (see Section 3).\vspace{6pt}\\
{\bfseries Theorem \ref{thm:instability}.}
 {\itshape  Compact complex \Cpf\ (or \Cp\ or \Cf\ or \pf\ or \p\ or \f) manifolds are not stable under small
 deformations of the complex structure.}
\vspace{6pt}

Furthermore, we show the following \vspace{6pt}\\
{\bfseries Theorem \ref{nilmanifold-no-tamed}.}
{\itshape Let $X\,:=\,\left(\Z\left[\im\right]\right)^3 \left\backslash\left(\C^3,\,*\right) \right.$
be the Iwasawa manifold. Then any small complex deformation of $X$ cannot be tamed by any symplectic form.}
\vspace{6pt}

As \Cpf\ property is defined for an arbitrary almost-complex structure (even not integrable), we study its
stability along curves of almost-complex structures $\left\{J_t\right\}_{t\in\left(-\varepsilon,\varepsilon\right)}$,
too. \newline
In \cite{draghici-li-zhang1} it is proved the semi-continuity property of $h^\pm_J$ for an almost-complex structure on
a compact $4$-dimensional manifold. More precisely, if $M$ is a compact $4$-manifold with an almost-complex structure
$J$ such that $\mathcal{K}^c_J\neq\emptyset$, then for any almost complex structure $J'$ in a sufficiently small
neighborhood of $J$ the following holds
\begin{enumerate}
\item[$\bullet$] $\mathcal{K}^c_{J'}\neq\emptyset$ \vskip.1truecm\noindent
\item[$\bullet$] $h^+_J(M)\leq h^+_{J'}(M)$ \vskip.1truecm\noindent
\item[$\bullet$] $h^-_J(M)\geq h^-_{J'}(M)$ \vskip.1truecm\noindent
\end{enumerate}

In \cite{lee}, curves of almost-complex structures parametrized by real forms of pure
degree $(2,0)+(0,2)$
are constructed. Using this construction, we prove that (see Theorem \ref{thm:n6(c)} for the precise statement)
there exists a family $\left\{N^6(c)\right\}_c$ of compact cohomologically K\"{a}hler manifolds with no K\"{a}hler
metrics such that
\begin{enumerate}
 \item[(i)] $N^6(c)$ admits a \Cpf\ almost-complex structure $J$,
 \item[(ii)] each harmonic form of type $(2,0)+(0,2)$ gives rise to a
 curve $\left\{J_t\right\}_{t\in\left(-\varepsilon,\varepsilon\right)}$ of \Cpf\ almost-complex
 structures on $N^6(c)$,
 \item[(iii)] furthermore, the map
 $$
 t\mapsto \dim_\R H^{(2,0),(0,2)}_{J_t}\left(N^6(c)\right)_\R
 $$
 is an
 upper-semicontinuous function at $t=0$.
\end{enumerate}
\vskip.1truecm\noindent
In particular, we get the upper-semicontinuity property of $h^-_J$ for this $6$-dimensional example.
\vskip.2truecm\noindent
We recall that, for a suitable $c\in\R$, the completely solvable Lie group
$$ \mathrm{Sol}(3) \;\definizione\;
\left\{
\left(
\begin{array}{cccc}
 \mathrm{e}^{cz} & 0 & 0 & x \\
 0 & \mathrm{e}^{-cz} & 0 & y \\
 0 & 0 & 1 & z \\
 0 & 0 & 0 & 1
\end{array}
\right)\in \GL(4;\,\R)
\,\,\,\vert\,\,\, x,y,z\in\R
\right\} $$
admits a cocompact discrete subgroup $\Gamma(c)$; we define
$$ M \;\definizione\; \Gamma(c) \left\backslash \mathrm{Sol}(3) \right. $$
and
$$ N^6(c) \;\definizione\; M \,\times\, M \;;$$
the manifold $N^6(c)$ first appeared in \cite{benson-gordon} as an example of a cohomologically
K\"{a}hler manifold; M. Fern\'{a}ndez, V. Mu\~{n}oz and J. A. Santisteban proved in
\cite{fernandez-munoz-santisteban} that it has no K\"{a}hler structures.

\vspace{12pt}
In Section \ref{sec:Cpf-structures}, we fix the notation, recall the main results on \Cpf\
almost-complex structures from \cite{li-zhang}, \cite{draghici-li-zhang} and \cite{fino-tomassini} and we give an example of
$6$-dimensional (compact) non-K\"ahler solvmanifold endowed with a \Cpf\  and \pf\ almost complex structure.
In Section \ref{sec:instability}, we prove the instability Theorem \ref{thm:instability}. Then, as
a consequence of a general fact (see Theorem \ref{omega-tamed}) true for $6$-dimensional nilmanifolds (i.e.
compact quotient of nilpotent symply-connected Lie groups by uniform dicrete subgroups), we
show that the deformed complex
structures described in Theorem \ref{thm:instability-iwasawa} cannot be tamed by any symplectic
form $\omega$ on the Iwasawa manifold. Finally,
in Section \ref{sec:almost-complex-curves}, we recall how a cohomology class in $H^{(2,0),(0,2)}_J(M)_\R$
gives rise to a curve of almost-complex structures on $\left(M,J\right)$; we provide several examples of
curves of \Cpf\ almost-complex structures on $4$ and $6$-dimensional compact manifolds.
\bigskip

\noindent This work has been originally developed as partial fulfillment for the first author's
Master Degree in Matematica Pura e Applicata at Universit\`{a} di Parma under the supervision of the second
author.\bigskip

\section{\Cpf\ almost-complex structures}\label{sec:Cpf-structures}
Let $\left(M,J\right)$ be an almost-complex compact $2n$-manifold.
The endomorphism $J$ on $TM\otimes\C$, having eigenvalues $\im$
and $-\im$, induces a decomposition of $\wedge^\bullet(M;\C)$
through $\wedge^{p,q}_JM \,:=:\, \wedge^{p,q}M$, namely
$\wedge^kM\,=\,\bigoplus_{p+q=k}\wedge^{p,q}_JM$. We ask about
when this decomposition holds in cohomology. Define
$H^{(p,q)}_J(M)$ as the projection in de Rham cohomology of the
space $\wedge^{p,q}M$; define $H^{(p,q),(q,p)}_J(M)_\R$ as the
projection in de Rham cohomology of the space
$\left(\wedge^{p,q}M\oplus\wedge^{q,p}M\right)\,\cap\,\wedge^{p+q}M$.
(As a matter of notation, bigraduation without further
specification refers to complex forms, single graduation to real
ones). In other words:
$$
H^{(p,q),(q,p)}_J(M)_\R \;=\; \left\{\left[\alpha\right]\in
H^{p+q}_{dR}(M;\R) \,\,\,\vert\,\,\,
\alpha\in\left(\wedge^{p,q}_JM\oplus\wedge^{q,p}_JM\right)\,\cap\,\wedge^{p+q}M\right\}.
$$
If $S$ is a set of pairs $(p,q)$, we define likewise
$$
H^S_J(M) \;\definizione\; \left\{\left[\alpha\right]\in H^\bullet_{dR}(M;\C) \,\,\,\vert\,\,\,
\alpha\in\bigoplus_{(p,q)\in S}\wedge^{p,q}M \right\}
$$
and
$$
H^S_J(M)_\R \;\definizione\; H^S_J(M) \,\cap\, H^\bullet_{dR}(M;\R) \;.
$$
T.-J. Li and W. Zhang give the following.
\begin{defi}[{\cite[Definition 2.2, Definition 2.3, Lemma 2.2]{li-zhang}}]\label{def:Cpf}
 An almost-complex structure $J$ on $M$ is said to be:
\begin{itemize}
 \item \emph{\Cp} if
$$ H^{(2,0),(0,2)}_J(M)_\R\;\cap\; H^{(1,1)}_J(M)_\R \;=\; \left\{\left[0\right]\right\} \;; $$
 \item \emph{\Cf} if
$$
H^{(2,0),(0,2)}_J(M)_\R\;+\; H^{(1,1)}_J(M)_\R \;=\; H^2_{dR}(M;\R) \;;
$$
 \item \emph{\Cpf} if it is both \Cp\ and \Cf, i.e. if the following decomposition holds:
$$
H^{(2,0),(0,2)}_J(M)_\R\,\oplus\, H^{(1,1)}_J(M)_\R \,=\, H^2_{dR}(M;\R) \;.
$$
\end{itemize}
For a complex manifold $M$, by saying that $M$ is, for example, \Cpf, we mean that the integrable
almost-complex structure naturally associated with it is \Cpf.
\end{defi}

We also use the following notations:
\begin{itemize}
\item by saying that $J$ is \emph{complex-\Cp} we mean that the sum
$$
H^{(2,0)}_J(M) \;+\; H^{(1,1)}_J(M) \;+\; H^{(0,2)}_J(M) $$
is direct;
 \item by saying that $J$ is \emph{complex-\Cf} we mean that the equality
$$ H^2_{dR}(M;\C) \;=\; H^{(2,0)}_J(M) \;+\; H^{(1,1)}_J(M) \;+\; H^{(0,2)}_J(M) $$
holds;
 \item by saying that $J$ is \emph{complex-\Cpf} we mean that $J$ induces the decomposition
$$ H^2_{dR}(M;\C) \;=\; H^{(2,0)}_J(M) \;\oplus\; H^{(1,1)}_J(M) \;\oplus\; H^{(0,2)}_J(M) \;.$$
\end{itemize}
\begin{rem}\label{complexCpf}
While being complex-\Cf\ is a stronger condition that being \Cf, one has to assume $J$ to
be integrable to have that complex-\Cp\ condition implies the \Cp\ one. Note also that if $J$ is \Cp\
then
$$
H^{(1,1)}_J(M)\,\cap\,\left(H^{(2,0)}_J(M) \,+\, H^{(0,2)}_J(M)\right)\,=\,\left\{\left[0\right]\right\}
\label{page:complex-Cp}.
$$
Moreover, we say that $J$ is \emph{\Cpf\ at the $k$-th stage} if $J$ induces a decomposition of
$H^k_{dR}(M;\R)$; for $k=2$, we recover the previous definitions.
\end{rem}
Using the complex of currents instead of the complex of forms and the de Rham homology instead of the de
Rham cohomology, one can define analogous concepts dually. Recall that the space of \emph{currents} of dimension
$k$ (or degree $2n-k$) is the topological dual of $\wedge^kM$: we denote it with $\correnti_kM\,:=:\,\correnti^{2n-k}M$;
we refer to \cite{derham}, \cite{demailly-agbook} as general references for the study of currents. Dually,
the exterior differential $\de$ on $\wedge^\bullet M$ induces a differential on $\correnti_\bullet M$, that
we denote again as $\de$; we call \emph{de Rham homology} $H_\bullet(M;\R)$ the cohomology of the
differential complex $\left(\correnti_\bullet M,\,\de\right)$; we remember that $H^k_{dR}(M;\R)\,\simeq\,H_{2n-k}(M;\R)$.
As $J$ induces a bigraduation on $\wedge^\bullet(M;\C)$, so $\correnti_{p,q}M$ are defined.

Therefore, let $H_{(2,0),(0,2)}^J(M)_\R$ (respectively, $H_{(1,1)}^J(M)_\R$) be the subspace of $H_2(M;\R)$
given by the homology classes represented by a current of bidimension $(2,0)+(0,2)$ (respectively, $(1,1)$).
We recall the following definition by T.-J. Li and W. Zhang (see \cite{li-zhang}).

\begin{defi}[{\cite[Definition 2.5, Lemma 2.7]{li-zhang}}]
 An almost-complex structure $J$ on $M$ is said to be:
\begin{itemize}
 \item \emph{\p} if
$$ H_{(2,0),(0,2)}^J(M)_\R\;\cap\; H_{(1,1)}^J(M)_\R \;=\; \left\{\left[0\right]\right\} \;; $$
 \item \emph{\f} if
$$ H_{(2,0),(0,2)}^J(M)_\R\;+\; H_{(1,1)}^J(M)_\R \;=\; H_2(M;\R)\;; $$
 \item \emph{\pf} if it is both \p\ and \f, i.e. if the following decomposition holds:
$$ H_{(2,0),(0,2)}^J(M)_\R\,\oplus\, H_{(1,1)}^J(M)_\R \,=\, H_2(M;\R) \;.$$
\end{itemize}
\end{defi}

The relations between being \Cpf\ and being \pf\ are summarized in the following.
\begin{thm}[{see also \cite[Proposition 2.5]{li-zhang}}]\label{thm:implicazioni}
 The following relations between \Cpf\ and \pf\ concepts hold:
$$
\xymatrix{
\text{\Cf\ at the }k\text{-th stage}\ar@{=>}[r]\ar@{=>}[d] & \text{\p\ at the }k\text{-th stage}\ar@{=>}[d] \\
\text{\f\ at the }\left(2n-k\right)\text{-th stage}\ar@{=>}[r] & \text{\Cp\ at the }\left(2n-k\right)\text{-th stage}
}
$$
\end{thm}

\begin{proof}
 First, we prove that if $J$ is \Cf\ at the $k$-th stage then it is also pure at the $k$-th stage; for
 the sake of simplicity, we assume $k=2$. Let
$$ \left\langle \sspace,\,\ssspace \right\rangle:\; H^2_{dR}(M;\R) \to H_2(M;\R) $$
the non-degenerate duality paring. Let
$\mathfrak{c}\in H_{(2,0),(0,2)}^J(M)_\R \,\cap\, H_{(1,1)}^J(M)_\R$, with $\mathfrak{c}\,\neq\,\left[0\right]$.
Obviously, $\left\langle \mathfrak{c},\, \sspace\right\rangle\lfloor_{H^{(2,0),(0,2)}_J(M)_\R}\,=\,0$
and $\left\langle \mathfrak{c},\, \sspace\right\rangle\lfloor_{H^{(1,1)}_J(M)_\R}\,=\,0$; since $J$ is \Cf,
it follows that $\mathfrak{c}\,=\,\left[0\right]$. The same argument works to prove that a \f\ $J$ is also \Cp.\\
To conclude the proof, we have to prove the two vertical arrows, namely that
$$ \text{\Cf\ at the }k\text{-th stage} \;\stackrel{\text{?}}{\Rightarrow}\;
\text{\f\ at the }\left(2n-k\right)\text{-th stage} $$
and that
$$ \text{\p\ at the }k\text{-th stage} \;\stackrel{\text{?}}{\Rightarrow}\;
\text{\Cp\ at the }\left(2n-k\right)\text{-th stage} \;.$$
Recall that a form of degree $k$ can be viewed as a current of dimension $2n-k$ (and degree $k$), by means of the map
$$ T_{\sspace}:\; \wedge^kM \to \correnti_{2n-k}M\;,\qquad \varphi \mapsto T_\varphi\left(\sspace\right)\;
\definizione\;\int_M \varphi \wedge \sspace \;.$$
Holding $T_{\de \sspace}\,=\,\de T_{\sspace}$, this map induces the inclusion $$H^{(p,q)}_J(M)_\R
\hookrightarrow H_{(n-p,n-q)}^J(M)_\R \;;$$ being $H^k_{dR}(M;\R)\,\simeq\, H_{2n-k}(M;\R)$, the statements follow.
\end{proof}

\noindent To prove that {\itshape \Cf\ $\Rightarrow$ \p}, see also \cite[Theorem 3.7]{fino-tomassini}.

A link between $H^2_{dR}(M;\R)$ and $H^{2n-2}_{dR}(M;\R)$ could provide further relations between \Cpf\ and
\pf\ notions. This is the matter of the following results, proved in \cite{fino-tomassini}.

\begin{thm}[{\cite[Theorem 3.7]{fino-tomassini}}]\label{thm:harmonic-representatives}
 Let $g$ be a Hermitian metric on $\left(M,J\right)$. If the harmonic representatives of the classes in
 $H^2_{dR}(M;\R)$ are of pure degree, then $J$ is both \Cpf\ and \pf.
\end{thm}

On a symplectic $2n$-manifold $\left( M,\omega \right)$, a link between $H^2_{dR}(M;\R)$ and $H^{2n-2}_{dR}(M;\R)$
could be provided if the Hard Lefschetz Condition holds; recall that $\left( M,\omega \right)$
is said to satisfy the \emph{Hard Lefschetz Condition} \eqref{eq:hlc} if, for every $k\in\left\{0,\ldots,n\right\}$,
the isomorphism
\begin{equation}\label{eq:hlc}
\left[\omega^k\right]  : \; H^{n-k}_{dR}(M;\R) \stackrel{\simeq}{\longrightarrow} H^{n+k}_{dR}(M;\R) \;.\tag{HLC}
\end{equation}
holds.

\begin{thm}[{\cite[Theorem 4.1]{fino-tomassini}}]
 Let $\omega$ be a symplectic form on $M$ satisfying \eqref{eq:hlc} and let $J$ an $\omega$-compatible
 almost-complex structure on $M$. If $J$ is \Cpf, then it is also \pf.
\end{thm}

We give now a class of examples of \Cpf\ and \pf\ manifolds. Clearly, compact K\"{a}hler manifolds are \Cpf\ at
every stage (and then also \pf\ at every stage). Furthermore, the following theorem holds.

\begin{thm}[{see \cite{li-zhang}, \cite{draghici-li-zhang}}]\label{thm:frolicher}
Let $M$ be a compact complex manifold; if the Hodge-Fr\"{o}licher spectral sequence degenerates at the first step and there is a weight $2$ formal Hodge decomposition,
then $M$ is \Cpf.
\end{thm}
% The proof of Theorem \ref{thm:frolicher} can be obtained using an argument like that one true for a compact complex surface as in
% \cite{draghici-li-zhang}, by using
% \begin{eqnarray*}
% H^2_{dR}(M;\R)\!\!\! &=& \!\!\!\left(H^{1,1}_{\delbar}(M)\cap H^2_{dR}(M;\R)\right)\oplus \left(\left(H^{2,0}_{\delbar}(M)\oplus H^{0,2}_{\delbar}(M)\right)\cap H^2_{dR}(M;\R)\right)\\[5pt]
% \!\!\! &=& \!\!\!  H^{(1,1)}_J(M)_\R \oplus H^{(2,0),(0,2)}_J(M)_\R
% \end{eqnarray*}
% (compare also \cite[Lemma 2.12]{draghici-li-zhang})
% instead of
% \begin{eqnarray*}
% H^2_{dR}(M;\C)\!\!\! &=& \!\!\! F^2H^2_{dR}(M;\C)\oplus \left(F^1H^2_{dR}(M;\C) \cap \overline{F^1H^2_{dR}(M;\C)}\right) \oplus \overline{F^2H^2_{dR}(M;\C)} \\[5pt]
% \!\!\! &=& \!\!\! H^{(2,0)}_J(M) \oplus H^{(1,1)}_J(M) \oplus H^{(0,2)}_J(M)
% \end{eqnarray*}
% (see \cite[Lemma 2.15]{draghici-li-zhang}); the argument for the other stages is similar (remember that being \Cf\ at every stage implies being \Cpf\ and \pf\ at every stage).

\noindent For the proof, see \cite[Theorem 2.16, Proposition 2.17]{draghici-li-zhang}.
\medskip

\noindent As a consequence of the last theorem, we have that
\begin{enumerate}
 \item\label{enum:examples-surfaces} the compact complex surfaces,
 \item\label{enum:examples-deldelbar} the compact complex manifolds satisfying the $\del\delbar$-Lemma
 (i.e., the compact complex manifolds for which every $\del$-closed, $\delbar$-closed and $\de$-exact
 form is also $\del\delbar$-exact)
 \item\label{enum:examples-kahler} and the compact complex manifolds admitting a K\"{a}hler structure
\end{enumerate}
are \Cpf\ manifolds.\\
Indeed, for \eqref{enum:examples-surfaces} we have that the assumptions of Theorem \ref{thm:frolicher} hold by \cite[Theorem 2.6]{barth-peters-vandeven}, while for \eqref{enum:examples-deldelbar}
they are satisfied by \cite[\S5.21]{deligne}; finally, for \eqref{enum:examples-kahler} we have that
a compact complex manifold admitting a K\"{a}hler metric satisfies the $\del\delbar$-Lemma, see \cite[\S5.11]{deligne}.

Actually, T. Dr\v{a}ghici, T.-J. Li and W. Zhang proved the following.
\begin{thm}[{\cite[Theorem 2.3]{draghici-li-zhang}}]
 Every almost-complex structure on a compact $4$-manifold is \Cpf\ as well as \pf.
\end{thm}

\noindent This turns our attention to the $6$-dimensional case.
\begin{ex}
Let $G$ be the Lie group of matrices of the following form
$$
A= \left(
\begin{array}{llllll}
e^{x_1}  &   0           &  x_2e^{x_1}        &  0            &  0        &  x_3\\
0        &  e^{-x_1}     &   0                &  x_2e^{-x_1}  &  0        &  x_4\\
0        &  0            &   e^{x_1}          &  0            &  0        &  x_5\\
0        &  0            &   0                &  e^{-x_1}      &  0       &  x_6\\
0        &  0            &   0                &  0             &  0       &  x_1\\
0        &  0            &   0                &  0             &  0       &  1
\end{array}
\right)
$$
for $x_1,\ldots ,x_6\in\R$. Then $G$ is a $6$-dimensional simply-connected completely solvable Lie group. According
to \cite{fernandez-deleon-saralegui}, there exists a uniform discrete subgroup $\Gamma\subset G$, so that
$M=\Gamma\backslash G$ is a $6$-dimensional compact solvmanifold. The following $1$-forms on $G$
$$
\begin{array}{lll}
e^1=dx_1\,,& e^2=dx_2\,,&e^3=\exp(-x_1)\left(dx_3-x_2dx_5\right)\\[10pt]
e^4=\exp(x_1)\left(dx_4-x_2dx_6\right)\,,&e^5=\exp(-x_1)dx_5\,,&e^6=\exp(x_1)dx_6
\end{array}
$$
give rise to $1$-forms on $M$. We immediately obtain that
\begin{equation}\label{solvablemaurercartanequations}
\begin{array}{lll}
\de e^1=0\,,& \de e^2=0\,,& \de e^3=-e^{1}\wedge e^3-e^{2}\wedge e^5\,,\\[10pt]
\de e^4=e^{1}\wedge e^4-e^{2}\wedge e^6\,,& \de e^5=-e^{1}\wedge e^5\,,& \de e^6=e^{1}\wedge e^5
\end{array}
\end{equation}
Since $G$ is completely solvable, in view of the
Hattori theorem (see \cite{hattori}), we easily obtain by \eqref{solvablemaurercartanequations}, that
\begin{equation}\label{solvablecohomology}
H^2(M;\R)\;=\; \Span_\R\left\{e^{1}\wedge e^{2},\, e^{5}\wedge e^{6},\, e^{3}\wedge e^{6}+e^{4}\wedge e^{5}\right\} \;;
\end{equation}
Therefore, setting
$$
\left\{
\begin{array}{l}
 \varphi^1 \;=\; e^1+\im e^2 \\[5pt]
 \varphi^2 \;=\; e^3+\im e^4 \\[5pt]
 \varphi^3 \;=\; e^5+\im e^6
\end{array}
\right. \;,
$$
we have that the almost complex structure $J$ whose complex forms of type $(1,0)$ are $\varphi^1, \varphi^2, \varphi^3$ is
\Cf. Indeed,
$$
\begin{array}{ll}
H^{(1,1)}_J\left(M\right)_\R\;&
=\; \Span_\R\left\{-\frac{1}{2i}\varphi^1\wedge\overline{\varphi}^1,\,
-\frac{1}{2i}\varphi^3\wedge\overline{\varphi}^3\right\}\,,\\[10pt]
H^{(2,0),(0,2)}_J\left(M\right)_\R\;&
=\; \Span_\R\left\{\frac{1}{2i}\left(\varphi^2\wedge\varphi^3-\overline{\varphi}^2\wedge\overline{\varphi}^3\right)\right\}\,.
\end{array}
$$
According to theorem \ref{thm:harmonic-representatives}, since the harmonic representatives are of pure type,
$J$ is both \Cpf\ and \pf.

\end{ex}
\section{Instability along curves of complex structures}\label{sec:instability}
In this section, we will show that the condition to be \Cpf\ for a complex structure is not
stable under small deformations. In order to do this, we will consider the Iwasawa manifold, showing that
there are curves of complex structures that are not \Cpf.

We first recall the definition of the Iwasawa manifold and some of its properties, see e. g. \cite{nakamura}, \cite{fernandez-gray}.

On $\C^3$, consider the product $*$ defined as
$$ \left( z_1, \; z_2, \; z_3 \right) \;*\; \left( w_1, \; w_2, \; w_3 \right) \;\definizione\;
\left( z_1+w_1, \; z_2+w_2, \; z_3+z_1 w_2+w_3 \right) \;. $$
It is immediate to check that $\left(\C^3,\,*\right)$ is a nilpotent Lie group isomorphic to
$$ \mathbb{H}(3) \;\definizione\;
\left\{\left(
\begin{array}{ccc}
 1 & z_1 & z_3 \\
 0  & 1     & z_2 \\
 0  & 0      & 1
\end{array}
\right)\in\GL\left(3;\,\C\right)\,\,\,
\vert \,\,\,z_1,z_2,z_3\in\C
\right\} \;.
$$
We have that $\left(\Z\left[\im\right]\right)^3\subset \C^3$ is a cocompact discrete subgroup of
$\left(\C^3,\,*\right)$. The \emph{Iwasawa manifold} $X$ is defined as the manifold
$$
X \;\definizione\;  \left(\Z\left[\im\right]\right)^3 \left\backslash\left(\C^3,\,*\right) \right.\;.
$$
$X$ is a compact complex $3$-dimensional nilmanifold; by \cite{fernandez-gray}, it follows that $X$ is not formal;
hence, it has no K\"{a}hler metrics, see \cite[Main Theorem]{deligne}; nevertheless, there exists a balanced metric on $X$.

We will need the following results on the cohomology of solvmanifolds. The Hattori-Nomizu
theorem states that if $M=\Gamma\backslash G$ is a compact nilmanifold (or, more in general, a compact completely solvable solvmanifold,
i.e. a compact solvmanifold such that, for every $\xi$ in the Lie algebra $\mathfrak{g}$ of $G$, all
the eigenvalues of $\mathrm{ad}_\xi$ are real) then
\begin{equation}\label{eq:hattori-nomizu}
H^\bullet_{dR}(M;\R) \;\simeq\; H^\bullet(\wedge^\bullet \mathfrak{g}^*,\, \de)
\end{equation}
(see \cite{nomizu}, \cite{hattori}), where the \emph{Chevalley-Eilenberg cohomology}
$H^\bullet(\wedge^\bullet \mathfrak{g}^*,\, \de)$ is the cohomology of the complex $\wedge^\bullet \mathfrak{g}^*$
endowed with the differential inherited from $\wedge^\bullet M$; equivalently,
$H^\bullet(\wedge^\bullet \mathfrak{g}^*,\, \de)$ is the cohomology of the complex of the left-invariant forms.
A similar result holds for the Dolbeault cohomology of nilmanifolds. More precisely, for a compact complex
nilmanifold $M$ that is holomorphically parallelizable (i.e., with trivial holomorphic tangent bundle) or whose
integrable almost-complex structure $J$ is rational (i.e., such that $J\left[\mathfrak{g}_\Q\right]\subseteq\mathfrak{g}_\Q$,
where $\mathfrak{g}_\Q$ is a rational Lie subalgebra of $\mathfrak{g}$ such that $\mathfrak{g}=\mathfrak{g}_\Q\otimes\R$)
or whose $J$ is obtained as a small deformation of a rational one, the following isomorphism holds:
\begin{equation}\label{eq:sakane}
H^{p,q}_{\delbar}(M) \;\simeq\; H^{q}\left(\wedge^{p,\bullet}\left(\mathfrak{g}^\C\right)^*,\,\delbar\right)
\end{equation}
(see \cite{sakane}, \cite{console-fino}). In particular, \eqref{eq:hattori-nomizu} and \eqref{eq:sakane} hold
for the Iwasawa manifold and for its small deformations.

Let $\left(z^i\right)_{i\in\{1,2,3\}}$ be the standard complex coordinate system on $\C^3$; the following $(1,0)$-forms
on $\C^3$ are invariant for the action (on the left) of $\left(\Z\left[\im\right]\right)^3$, so they give rise to a
global coframe for $T^{*\,1,0}X$:
$$
\left\{
\begin{array}{l}
 \varphi^1\;\definizione\;\de z^1\\[5pt]
 \varphi^2\;\definizione\;\de z^2\\[5pt]
 \varphi^3\;\definizione\;\de z^3-z^1\de z^2
\end{array}
\right. \;.
$$
The structure equations are therefore
$$
\left\{
\begin{array}{l}
 \de\varphi^1\;=\;0\\[5pt]
 \de\varphi^2\;=\;0\\[5pt]
 \de\varphi^3\;=\;-\varphi^1\wedge\varphi^2
\end{array}
\right. \;.
$$
By Hattori-Nomizu theorem, we compute the real cohomology algebra of $X$ (for simplicity, we list the harmonic
representative instead of its class and write $\varphi^{A\bar{B}}$ for $\varphi^A\wedge\bar{\varphi}^B$ and so on):
\begin{eqnarray*}
 H^1_{dR}(X;\R) &=& \Span_\R\left\{\varphi^1+\bar{\varphi}^1,\; \im\left(\varphi^1-\bar{\varphi}^1\right),\;
 \varphi^2+\bar{\varphi}^2,\; \im\left(\varphi^2-\bar{\varphi}^2\right)\right\} \;,\\[10pt]
 H^2_{dR}(X;\R) &=& \Span_\R\Big\{\varphi^{13}+\varphi^{\bar{1}\bar{3}},\; \im
 \left(\varphi^{13}-\varphi^{\bar{1}\bar{3}}\right),\;\varphi^{23}+\varphi^{\bar{2}\bar{3}},\\[5pt]
&& \im\left(\varphi^{23}-\varphi^{\bar{2}\bar{3}}\right),\;\varphi^{1\bar{2}}-\varphi^{2\bar{1}},\;
\im\left(\varphi^{1\bar{2}}+\varphi^{2\bar{1}}\right),\;\im\varphi^{1\bar{1}},\;\im\varphi^{2\bar{2}}\Big\} \;,\\[10pt]
 H^3_{dR}(X;\R) &=& \Span_\R\Big\{\varphi^{123}+\varphi^{\bar{1}\bar{2}\bar{3}},\;
 \im\left(\varphi^{123}-\varphi^{\bar{1}\bar{2}\bar{3}}\right),\;\varphi^{13\bar{1}}+\varphi^{1\bar{1}\bar{3}},\\[5pt]
&& \im\left(\varphi^{13\bar{1}}-\varphi^{1\bar{1}\bar{3}}\right),\; \varphi^{13\bar{2}}+\varphi^{2\bar{1}\bar{3}},\;
\im\left(\varphi^{13\bar{2}}-\varphi^{2\bar{1}\bar{3}}\right),\;\varphi^{23\bar{1}}+\varphi^{1\bar{2}\bar{3}},\\[5pt]
&&\im\left(\varphi^{23\bar{1}}-\varphi^{1\bar{2}\bar{3}}\right),\; \varphi^{23\bar{2}}+\varphi^{2\bar{2}\bar{3}},\;
\im\left(\varphi^{23\bar{2}}-\varphi^{2\bar{2}\bar{3}}\right)\Big\} \;.
\end{eqnarray*}
Note that each harmonic representative is of pure degree. The Betti numbers of $X$ are
$$ b^0 \,=\, 1 \;, \quad b^1 \,=\,  4 \;, \quad b^2 \,=\, 8 \;, \quad b^3 \,=\, 10\;. $$

I. Nakamura in \cite{nakamura} computed the small deformations
$\left\{X_\mathbf{t}\right\}_{\mathbf{t}\,\in\,\Delta\left(\mathbf{0},\,\varepsilon\right)\,\subseteq\,\C^6}$ of
the Iwasawa manifold $X$: by \cite[page 95]{nakamura}, a local system of complex coordinates for the complex structure
at $\mathbf{t}\,=\,\left(t_{11},\,t_{12},\,t_{21},\,t_{22},\,t_{31},\,t_{32}\right)\in\C^6$ is given by
$$
\left\{
\begin{array}{rcl}
 \zeta^1_\mathbf{t} &=& z^1\,+\,\sum_{j=1}^{2}t_{1j}\,\bar{z}^j \\[5pt]
 \zeta^2_\mathbf{t} &=& z^2\,+\,\sum_{j=1}^{2}t_{2j}\,\bar{z}^j \\[5pt]
 \zeta^3_\mathbf{t} &=& z^3\,+\,\sum_{j=1}^{2}\left(t_{3j}+t_{2j}\,z^1\right)\bar{z}^j\,+
 \,A(\bar{\mathbf{z}})\,-\,D(\mathbf{t})\,\bar{z}^3
\end{array}
\right.
$$
where
\begin{eqnarray*}
A(\bar{\mathbf{z}})&\definizione& \frac{1}{2}\left(t_{11}\,t_{21}\,\bar{z}^1\,\bar{z}^1+
2\,t_{11}\,t_{22}\,\bar{z}^1\,\bar{z}^2+t_{12}\,t_{22}\,\bar{z}^2\,\bar{z}^2\right)\;,\\[5pt]
D(\mathbf{t})&\definizione& t_{11}\,t_{22}-t_{12}\,t_{21} \;.
\end{eqnarray*}
Nakamura also computed the numerical characters of these deformations, dividing them into three classes according to
their Hodge diamond:
\vspace{12pt}
\begin{center}
\begin{tabular}{c||cc|ccc|cccc}
 & $h^{1,0}$ & $h^{0,1}$ & $h^{2,0}$ & $h^{1,1}$ & $h^{0,2}$ & $h^{3,0}$ & $h^{2,1}$ & $h^{1,2}$ & $h^{0,3}$ \\[2pt]
\hline
{\itshape (i)} & $3$ & $2$ & $3$ & $6$ & $2$ & $1$ & $6$ & $6$ & $1$ \\[5pt]
{\itshape (ii)} & $2$ & $2$ & $2$ & $5$ & $2$ & $1$ & $5$ & $5$ & $1$ \\[5pt]
{\itshape (iii)} & $2$ & $2$ & $1$ & $5$ & $2$ & $1$ & $4$ & $4$ & $1$
\end{tabular}
\end{center}
\vspace{12pt}
More exactly, the classes are characterized by the following values of the parameters:
\begin{description}
 \item[class (i)] $t_{11}=t_{12}=t_{21}=t_{22}=0$;\\
 \item[class (ii)] $D\left(\mathbf{t}\right)=0$ but $\left(t_{11},\,t_{12},\,t_{21},\,t_{22}\right)\neq
 \left(0,\,0,\,0,\,0\right)$;\\
 \item[class (iii)] $D\left(\mathbf{t}\right)\neq 0$.
\end{description}
Note that the Hodge diamond of the deformations of the class (i) is the same of the Iwasawa manifold,
while deformations of the class (iii) have the Hodge-Fr\"{o}licher spectral sequence that degenerates at
the first step. Note also that the table above proves that the Hodge numbers are not stable under small
deformations, \cite[Theorem 2]{nakamura}, in contrast with the K\"{a}hler case.

Equivalently, $X_\mathbf{t}$ could be viewed as $\left.\C^3\right/ \Gamma_\mathbf{t}$ where $\Gamma_\mathbf{t}$
is the group generated by the transformations
$$
\left\{
\begin{array}{rcl}
 \zeta'^1 & := & \zeta^1+\left(\omega_1+\sum_{j=1}^{2}t_{1j}\,\bar{\omega}_j\right) \\[10pt]
 \zeta'^2 & := & \zeta^2+\left(\omega_2+\sum_{j=1}^{2}t_{2j}\,\bar{\omega}_j\right) \\[10pt]
 \zeta'^3 & := & \zeta^3+\left(\omega_3+\sum_{j=1}^{2}t_{3j}\,\bar{\omega}_j\right)+\omega_1\,\zeta^2+
 \left(\sum_{j=1}^{2}t_{2j}\,\bar{\omega}_j\right)\,\left(\zeta^1+\omega_1\right)\\[10pt]
&\;&+A\left(\bar{\mathbf{\omega}}\right)-D(\mathbf{t})\,\bar{\omega}_3
\end{array}
\right.
$$
varying $\mathbf{\omega}:=:\left(\omega_1,\omega_2,\omega_3\right)\in\left(\Z[\im]\right)^3$.

In the sequel, $J$ will denote the integrable almost-complex structure associated to $X$ and $J_\mathbf{t}$
will denote the one associated to $X_\mathbf{t}$.

Now, we can prove the following.
\begin{thm}\label{thm:instability-iwasawa}
 Let $X\,:=\,\left(\Z\left[\im\right]\right)^3 \left\backslash\left(\C^3,\,*\right) \right.$
be the Iwasawa manifold. Then:
\begin{itemize}
 \item $X$ is \Cpf\ at every stage as well as \pf\ at every stage;
 \item the small deformations of $X$ of the class (i) are \Cpf\ and \pf\ at every stage;
 \item the small deformations of $X$ of the classes (ii) and (iii) are neither \Cp\ nor \Cf\ nor \p\ nor \f.
\end{itemize}
\end{thm}

\begin{proof}
 We divide the proof in various steps.
\begin{description}
 \item[Step 1] {\itshape $X$ is a \Cpf\ manifold at every stage.}\\
Since harmonic representatives in $H^\bullet_{dR}(X;\R)$ are of pure degree, the statement follows
from Theorem \ref{thm:harmonic-representatives}.
 \item[Step 2] {\itshape Small deformations of the class (i) remain \Cpf\ at every stage.}\\
A coframe of $(1,0)$-forms invariant for the action of $\Gamma_\mathbf{t}$ on $\C^3$ is given by
$$
\left\{
\begin{array}{rcl}
 \phit{1} &\definizione& \de\zeta^1_{\mathbf{t}} \\[5pt]
 \phit{2} &\definizione& \de\zeta^2_{\mathbf{t}} \\[5pt]
 \phit{3} &\definizione& \de\zeta^3_{\mathbf{t}}-\zeta^1_{\mathbf{t}}\de\zeta^2_{\mathbf{t}}
\end{array}
\right. \;.$$
Hence, $\left\{\phit{1},\phit{2},\phit{3}\right\}$ satisfies the same structure equations as
$\left\{\varphi^1,\varphi^2,\varphi^3\right\}$. Therefore, the same argument in {\itshape Step 1}
applies to deformations of such a class.
 \item[Step 3] {\itshape Computation of the structure equations for small deformations of the class (ii).}\\
Consider the system of complex coordinates given by
$$
\left\{
\begin{array}{rcl}
 \zeta^1_\mathbf{t} &\definizione& z^1 + \sum_{\lambda=1}^{2}t_{1\lambda}\bar{z}^{\lambda} \\[10pt]
 \zeta^2_\mathbf{t} &\definizione& z^2 + \sum_{\lambda=1}^{2}t_{2\lambda}\bar{z}^{\lambda} \\[10pt]
 \zeta^3_\mathbf{t} &\definizione& z^3 + \sum_{\lambda=1}^{2}(t_{3\lambda}+t_{2\lambda}z^1)\bar{z}^{\lambda}+
 A\left(\bar{z}\right)
\end{array}
\right. \;.
$$
A straightforward computation gives
$$
\left\{
\begin{array}{rcl}
 z^1 &=& \gamma\,\left(\zeta^1_\mathbf{t}+\lambda_1\bar{\zeta}^1_\mathbf{t}+\lambda_2\zeta^2_\mathbf{t}+
 \lambda_3\bar{\zeta}^2_\mathbf{t}\right) \\[10pt]
 z^2 &=& \alpha\,\left(\mu_0\zeta^1_\mathbf{t}+\mu_1\bar{\zeta}^1_\mathbf{t}+\mu_2\zeta^2_\mathbf{t}+
 \mu_3\bar{\zeta}^2_\mathbf{t}\right)
\end{array}
\right.
$$
where $\alpha$, $\beta$, $\gamma$, $\lambda_i$ (for $i\in\{1,2,3\}$), $\mu_j$ (for $j\in\{0,1,2,3\}$)
are complex constants depending only on
$\mathbf{t}\,=\,\left(t_{11},\,t_{12},\,t_{21},\,t_{22},\,t_{31},\,t_{32}\right)\in\C^6$
and defined as follows:
$$
\left\{
\begin{array}{rcl}
 \alpha &\definizione& \displaystyle \frac{1}{1-\modulo{t_{22}}^2-t_{21}\bar{t}_{12}}\\[15pt]
 \beta &\definizione& \displaystyle t_{21}\bar{t}_{11}+t_{22}\bar{t}_{21} \\[15pt]
 \gamma &\definizione& \displaystyle \frac{1}{1-\modulo{t_{11}}^2-\alpha\beta \left(t_{11}\bar{t}_{12}+
 t_{12}\bar{t}_{22}\right)-t_{12}\bar{t}_{21}} \\[15pt]
 \lambda_1 &\definizione& -t_{11}\left(1+\alpha\bar{t}_{12}t_{21}+\alpha\modulo{t_{22}}^2\right) \\[15pt]
 \lambda_2 &\definizione& \displaystyle \alpha \left(t_{11}\bar{t}_{12}+ t_{12}\bar{t}_{22}\right) \\[15pt]
 \lambda_3 &\definizione& \displaystyle -t_{12}\left(1+\alpha\bar{t}_{12}t_{21}+\alpha\modulo{t_{22}}^2\right) \\[15pt]
 \mu_0 &\definizione& \displaystyle \beta\gamma \\[15pt]
 \mu_1 &\definizione& \displaystyle \lambda_1\beta\gamma-t_{21} \\[15pt]
 \mu_2 &\definizione& \displaystyle 1+\lambda_2\beta\gamma \\[15pt]
 \mu_3 &\definizione& \displaystyle \lambda_3\beta\gamma-t_{22}
\end{array}
\right. \;.
$$
Consider the $(1,0)$-forms invariant for the action of $\Gamma_\mathbf{t}$ on $\C^3$ given by
$$
\left\{
\begin{array}{rcl}
 \phit{1} &\definizione& \de\zeta^1_\mathbf{t} \\[5pt]
 \phit{2} &\definizione& \de\zeta^2_\mathbf{t} \\[5pt]
 \phit{3} &\definizione& \de\zeta^3_\mathbf{t}-z^1\de\zeta^2_\mathbf{t}-\left(t_{21}\bar{z}^1+
 t_{22}\bar{z}^2\right)\de\zeta^1_\mathbf{t}
\end{array}
\right. \;;
$$
we could now easily compute the structure equations:
$$
\left\{
\begin{array}{rcl}
 \de\phit{1} &=& 0 \\[5pt]
 \de\phit{2} &=& 0 \\[5pt]
 \de\phit{3} &=& \sigma_{12}\,\pphit{1}{2}+\\[5pt]
&&+\sigma_{1\bar{1}}\,\pbphit{1}{1}+\sigma_{1\bar{2}}\,\pbphit{1}{2}+\sigma_{2\bar{1}}\,\pbphit{2}{1}+
\sigma_{2\bar{2}}\,\pbphit{2}{2}
\end{array}
\right.
$$
where $\sigma_{12}$, $\sigma_{1\bar{1}}$, $\sigma_{1\bar{2}}$, $\sigma_{2\bar{1}}$, $\sigma_{2\bar{2}}$ are given by
$$
\left\{
\begin{array}{rcl}
 \sigma_{12} &\definizione& -\gamma+t_{21}\bar{\lambda}_3\bar{\gamma}+t_{22}\bar{\alpha}\bar{\mu}_3 \\[10pt]
 \sigma_{1\bar{1}} &\definizione& t_{21}\,\overline{\gamma\left(1+t_{21}\bar{t}_{12}\alpha+
 \modulo{t_{22}}^2\alpha\right)} \\[10pt]
 \sigma_{1\bar{2}} &\definizione& t_{22}\,\overline{\gamma\left(1+t_{21}\bar{t}_{12}\alpha+
 \modulo{t_{22}}^2\alpha\right)} \\[10pt]
 \sigma_{2\bar{1}} &\definizione& -t_{11}\,\gamma\left(1+t_{21}\bar{t}_{12}\alpha+
 \modulo{t_{22}}^2\alpha\right) \\[10pt]
 \sigma_{2\bar{2}} &\definizione& -t_{12}\,\gamma\left(1+t_{21}\bar{t}_{12}\alpha+
 \modulo{t_{22}}^2\alpha\right)
\end{array}
\right. \;.
$$
Note that, for small deformations of the class (ii), one has $\sigma_{12}\,\neq\, 0$ and
$\left(\sigma_{1\bar{1}},\,\sigma_{1\bar{2}},\,\sigma_{2\bar{1}},\,\sigma_{2\bar{2}}\right)\,\neq\,
\left(0,\,0,\,0,\,0\right)$.\\
This ends {\itshape Step 3}.
 \item[Step 4] {\itshape The small deformations of the class (ii) are neither \Cp\ nor \f.}\\
Note that
$$ \left[\sigma_{12}\,\phit{12}\right] \;=\; \left[\sigma_{1\bar{1}}\,\phit{1\bar{1}}+
\sigma_{1\bar{2}}\,\phit{1\bar{2}}+\sigma_{2\bar{1}}\,\phit{2\bar{1}}+\sigma_{2\bar{2}}\,\phit{2\bar{2}}\right] \;
\neq\; \left[0\right] $$
in $H^2_{dR}\left(X;\C\right)$. Therefore,
$$ H^{(1,1)}_{J_\mathbf{t}}\left(X_{\mathbf{t}}\right)\,\cap\,\left(H^{(2,0)}_{J_\mathbf{t}}\left(X_{\mathbf{t}}\right) \,
+\, H^{(0,2)}_{J_\mathbf{t}}\left(X_{\mathbf{t}}\right)\right)\,\neq\,\left\{\left[0\right]\right\} \;, $$
and in particular $X_\mathbf{t}$ is not complex-\Cp. It follows from the fact observed at page \pageref{page:complex-Cp}
that $X_\mathbf{t}$ cannot be \Cp; from Theorem \ref{thm:implicazioni} it follows that $X_\mathbf{t}$ cannot be even \f.
 \item[Step 5] {\itshape The small deformations of the class (ii) are neither \p\ nor \Cf.}\\
For a fixed small $\mathbf{t}$, choose two positive constants $A$ and $B$ such that
$$ \left(A\,\sigma_{1\bar{2}}-B\,\sigma_{1\bar{1}},\,A\,\sigma_{2\bar{2}}-B\,\sigma_{2\bar{1}}\right) \;
\neq\; \left(0,\,0\right) \;;$$
computing $-\de\left(A\,\varphi_{\mathbf{t}}^{13\bar{3}}+B\,\varphi_{\mathbf{t}}^{23\bar{3}}\right)$,
note that
\begin{eqnarray*} \left[\left(A\,\sigma_{2\bar{1}}-B\,\sigma_{1\bar{1}}\right)\phit{12\bar{1}\bar{3}}+
\left(A\,\sigma_{2\bar{2}}-B\,\sigma_{1\bar{2}}\right)\phit{12\bar{2}\bar{3}}
-A\,\bar{\sigma}_{12}\phit{13\bar{1}\bar{2}}-B\,\bar{\sigma}_{12}\phit{23\bar{1}\bar{2}}\right] =\\[5pt]
=\left[\left(A\,\bar{\sigma}_{1\bar{2}}-B\,\bar{\sigma}_{1\bar{1}}\right)\phit{123\bar{1}}+
\left(A\,\bar{\sigma}_{2\bar{2}}-B\,\bar{\sigma}_{2\bar{1}}\right)\phit{123\bar{2}}\right]\neq\left[0\right] ,
\end{eqnarray*}
in $H^4_{dR}\left(X;\C\right)$. As before, it follows that $X_\mathbf{t}$ is not \Cp\ at the $4$th stage,
and consequently it is not even \p\ nor \Cf, by Theorem \ref{thm:implicazioni}.
 \item[Step 6] {\itshape Small deformations of the class (iii) are neither \Cp\ nor \Cf.}\\
We omit the computations, since they are quite similar to those ones corresponding to the deformations of class (ii).
\qedhere
\end{description}
\end{proof}
As a corollary of the last theorem, we obtain the following theorem of instability.
\begin{thm}\label{thm:instability}
Compact complex \Cpf\ (or \Cp\ or \Cf\ or \pf\ or \p\ or \f) manifolds are not stable under small
deformations of the complex structure.
\end{thm}
\begin{rem}
By the proof of Theorem \ref{thm:instability-iwasawa}, it follows that the numbers
$$
h^+_{J_{\bf t}}(X)\;\definizione\; \dim_\R H^{(1,1)}_{J_{\bf t}}(X)_\R\,,\qquad
h^-_{J_{\bf t}}(X)\;\definizione\; \dim_\R H^{(2,0),(0,2)}_{J_{\bf t}}(X)_\R\,,
$$
where $J_{\bf t}$ is a small deformation of class (i), are equal to
% $$
% h^+_J(X)=h^-_J(X)=4\,,
% $$
% and
$$
h^+_{J_{\bf t}}(X)=4\,,\quad h^-_{J_{\bf t}}(X)=4\,,
$$
for ${\bf t}$ small enough.\newline
We recall that for the complex deformations of a complex structure $J$ on a $4$-dimensional compact
manifold $M$, one has that $h^+_J(M)$ and $h^-_J(M)$ are topological invariants (see \cite{draghici-li-zhang1}).
\end{rem}
Now we show that none of the above deformed complex structures described in
Theorem \ref{thm:instability-iwasawa} can be tamed by any symplectic form $\omega$ on the Iwasawa manifold.

We start with the following
\begin{prop}\label{d-dbar}
Let $(M,\omega)$ be a symplectic manifold. Assume that there
exists an $\omega$-tamed complex structure $J$ on $M$. Denote
by $\tilde\omega$ the fundamental form of the Hermitian metric
$$
\tilde{g}_J(X,Y)=\frac{1}{2}\left(\omega(X,JY)+
\omega(Y,JX)\right)\,.
$$
Then $$
\partial\overline{\partial}\,\tilde\omega =0
$$
\end{prop}
\begin{proof}
By definition, we have
$$
\tilde{\omega} =\frac{1}{2}\left(\omega+ J\omega\right)\,.
$$
Therefore, by viewing $\omega$ and $\tilde{\omega}$ as real
elements of $\wedge^2(M;\C)$, we can write
\begin{eqnarray*}
\omega &=& \omega^{2,0} +
\omega^{1,1}+\overline{\omega^{2,0}}\,,\\
\tilde{\omega} &=&\omega^{1,1}\,,
\end{eqnarray*}
where $\overline{\omega^{1,1}}=\omega^{1,1}$. Hence, writing
$\de =\partial +\overline{\partial}$, we have that
$$
\de\omega = 0\,\quad\Leftrightarrow \,\quad \left\{
\begin{array}{rl}
\partial\, \omega^{2,0} & = 0\,,\\[10pt]
\partial\,\omega^{1,1} + \overline{\partial}\,\omega^{2,0} &=0\,.
\end{array}
\right.
$$
Therefore,
$$
\partial\overline{\partial}\,\tilde{\omega}=\partial\overline{\partial}\,\omega^{1,1}=
-\overline{\partial}\partial\,\omega^{1,1}=\overline{\partial}^2\,\omega^{2,0}=0\,.
$$
\end{proof}
%\begin{rem} Hermitian metrics $g$ on a complex manifold $(M,J)$ whose fundamental form $\omega$ satisfies the
%equation $\partial\overline{\partial}\omega =0$ are called in the literature \emph{strong K\"ahler with
%torsion metrics}.
%\end{rem}
Now, we can prove the following
\begin{thm} \label{omega-tamed}
Let $M=\Gamma\backslash G$ be a $($non-toral$)$ compact nilmanifold of
dimension $6$ endowed with an invariant complex structure $J$.
Then there are no symplectic structures $\omega$ on $M$ such that
$J$ is $\omega$-tamed.
\end{thm}
\begin{proof} First of all we show that there are no invariant symplectic forms $\omega$
on $M$ such that $J$ is $\omega$-tamed. On
the contrary, let $\omega$ be such an invariant symplectic
structure. Then by Proposition \ref{d-dbar}, $M$ has an invariant Hermitian metric $g$
whose fundamental form $\omega$ satisfies $\partial\overline{\partial}\,\omega=0$. Then,
by \cite[Theorem 1.2 p.
320]{FPS} we can find a basis of complex $(1,0)$-forms
$\{\varphi^1,\varphi^2,\varphi^3\}$ for $J$ such that
\begin{equation*}\label{nilpotent}
\left\{
\begin{array}{l}
\de\varphi^1=0\,, \\[7pt]
\de\varphi^{2}=0\,,\\[7pt]
\de\varphi^3=A\,\overline{\varphi}^1\wedge\varphi^2+B\,\overline{\varphi}^2\wedge\varphi^2+
C\,\varphi^1\wedge\overline{\varphi}^1+
D\,\varphi^1\wedge\overline{\varphi}^2+E\,\varphi^1\wedge\varphi^2\,,
\end{array}
\right.
\end{equation*}
where $A,\,B,\,C,\,D,\,E\in \C$. We immediately obtain
\begin{equation}\label{nilpotentd-bar}
\begin{array}{lll}
\partial\,\varphi^1=0\,, &\partial\,\varphi^{2}=0\,,&\partial\,\varphi^3=E\,\varphi^1\wedge\varphi^2\,,\\[5pt]
\overline{\partial}\,\varphi^1=0\,,& \overline{\partial}\,\varphi^{2}=0\,,&
\overline{\partial}\,\varphi^3=A\,\overline{\varphi}^1\wedge\varphi^2+
B\,\overline{\varphi}^2\wedge\varphi^2+C\,\varphi^1\wedge\overline{\varphi}^1+
D\,\varphi^1\wedge\overline{\varphi}^2.
\end{array}
\end{equation}
Write
$$
\omega =\omega^{2,0}+\omega^{1,1}+\overline{\omega^{2,0}}\,,
$$
where
$$
\omega^{2,0}=\sum_{i<j}a_{ij}\,\varphi^i\wedge\varphi^j\,,
\quad\omega^{1,1}=\frac{i}{2}\sum_{i,j=1}^3b_{i\,\overline{j}}\,\varphi^{i}\wedge\overline{\varphi}^j\,,
\quad \omega^{1,1}=\overline{\omega^{1,1}}\,,
$$
with $a_{ij},\,b_{i\,\overline{j}}\in\C$. Then a
straightforward computation by using \eqref{nilpotentd-bar} implies that
$\de\omega=0$ if and only if
$$
A=B=C=D=E=0
$$
or
$$
b_{3\overline{3}}=0\,.
$$
In both cases we obtain a contradiction.\smallskip

Assume that there exists a symplectic form
$\omega$ on $M$ such that $J$ is $\omega$-tamed. Then one can
consider
\begin{equation}
\label{simpletticamedia}
\hat{\omega}(X,Y)=\int_M\omega_{\vert_p}(X_p,Y_p)\eta\,.
\end{equation}
$\eta$ being the volume form given in \cite[lemma 6.2]{M}. Then, it can be showed (see e.g.
\cite{B, FG}) that $\de\hat{\omega}=0$. Moreover, since $J$ is invariant, we have
\begin{eqnarray*}
\hat{\omega}(X,JX) &=& \int_M\omega_{\vert_p}(X_p,(JX)_p)\eta\\
&=&\int_M\omega_{\vert_p}(X_p,J_pX_p)\eta
>0\,.
\end{eqnarray*}
Therefore
\eqref{simpletticamedia} defines an invariant symplectic form
$\hat{\omega}$ such that $J$ is $\hat{\omega}$-tamed. This is
absurd.
\end{proof}

Relating to the speculation in \cite[p. 678]{li-zhang} by T.-J. Li and W. Zhang and to \cite[Question 1.7]{ST} raised by J. Streets and G. Tian, as a consequence of the last
Theorem, we get the following
\begin{thm}\label{nilmanifold-no-tamed}
Let $X\,:=\,\left(\Z\left[\im\right]\right)^3 \left\backslash\left(\C^3,\,*\right) \right.$
be the Iwasawa manifold. Then any small complex deformation of $X$ cannot be tamed by any symplectic form.
\end{thm}
\section{Stability along curves of almost-complex structures}\label{sec:almost-complex-curves}
The \Cpf\ property makes sense for an arbitrary almost-complex structure, even not integrable. In this
section, we will study the stability of this property along curves of almost-complex structures.

Let $J$ be an almost-complex structure on $M^{2n}$. Recall the following local result, see \cite{audin-lafontaine}:
a curve $\left\{J_t\right\}_{t\in I\subseteq \R}$ of almost-complex structures through $J$ could be written,
for $t\in\left(-\varepsilon,\varepsilon\right)$ with $\varepsilon>0$ sufficiently small, as
\begin{equation}\label{eq:Jt}
J_t \;=\; \left(\id\,-\,L_t\right)\,J\,\left(\id\,-\,L_t\right)^{-1} \;,
\end{equation}
where $L_t$ is an endomorphism of the tangent bundle such that
$$ L_t\,J\,+\,J\,L_t \;=\; 0 \;;$$
we could write $L_t\,=:\, t\,L+\mathrm{o}(t)$; recall also that: if $J$ is compatible with a symplectic
form $\omega$, then the curves made up of $\omega$-compatible almost-complex structures $J_t$ are exactly
those ones for which $\trasposta{L}\,=\,L$. For several examples of families constructed in such a way,
see \cite{fino-tomassini}.

We begin with studying curves through the standard K\"{a}hler structure on the complex $2$-torus,
$\left(\T^2_\C,J_0,\omega_0\right)$. Let
$$ L\;=\;\left(
\begin{array}{cc|cc}
 \phantom{+}\ell &  &  &  \\
 & \phantom{+}0 &  &  \\
\hline
  & & -\ell &\\
 &  &  & \phantom{+}0
\end{array}
\right)
\;,$$
where $\ell\in\mathcal{C}^\infty(\R^4;\,\R)$ is a $\Z^4$-periodic function. Defining (for small $t$)
\begin{equation}\label{eq:Jt-L}
 J_{t,\,\ell} \;\definizione\; \left(\id\,-\,t\,L\right)\,J\,\left(\id\,-\,t\,L\right)^{-1} \;,
\end{equation}
we get the $\omega_0$-compatible almost-complex structure
$$ J_{t,\,\ell}\;=\;
\left(
\begin{array}{cc|cc}
 & & -\frac{1\,-\,t\,\ell}{1\,+\,t\,\ell} & \\
 &  &  & -1 \\
\hline
 \phantom{+}\frac{1\,+\,t\,\ell}{1\,-\,t\,\ell} &  &  & \\
 & \phantom{+}1 &  & \\
\end{array}
\right) \;.
$$
Set
$$ \alpha\;:=:\;\alpha(t,\ell)\;\definizione\; \frac{1\,-\,t\,\ell}{1\,+\,t\,\ell} \;. $$
A coframe for the holomorphic cotangent bundle is given by
$$
\left\{
\begin{array}{l}
 \varphi^1_t \;:=\; \de x^1\,+\,\im\,\alpha\,\de x^3 \\[5pt]
 \varphi^2_t \;:=\; \de x^2\,+\,\im\,\de x^4
\end{array}
\right. \;,
$$
from which we compute
$$
\left\{
\begin{array}{l}
 \de\varphi^1_t \;=\; \im\,\de\alpha\,\wedge\,\de x^3 \\[5pt]
 \de\varphi^2_t \;=\; 0
\end{array}
\right. \;;
$$
note that taking $\ell\,=\,\ell\left(x^1,x^3\right)$, the corresponding almost-complex structure $J_{t,\,\ell}$
is integrable, in fact $\left(J_{t,\,\ell},\,\omega_0\right)$ is a K\"{a}hler structure on $\T^2_\C$.
Remember that $J_{t,\,\ell}$ has to be \Cpf, $\T_\C^2$ being a $4$-dimensional manifold. For the sake of
simplicity, we assume $\ell\,=\,\ell\left(x^2\right)$ not constant.
Setting
\begin{eqnarray*}
 v_1 &\definizione& \de x^1\wedge \de x^2 - \alpha\, \de x^3 \wedge \de x^4 \;, \\[5pt]
 v_2 &\definizione& \de x^1\wedge \de x^4 - \alpha\, \de x^2 \wedge \de x^3 \;, \\[10pt]
 w_1 &\definizione& \alpha\,\de x^1\wedge \de x^3 \;, \\[5pt]
 w_2 &\definizione& \de x^2\wedge \de x^4 \;, \\[5pt]
 w_3 &\definizione& \de x^1\wedge \de x^2 + \alpha\, \de x^3 \wedge \de x^4 \;, \\[5pt]
 w_4 &\definizione& \de x^1\wedge \de x^4 + \alpha\, \de x^2 \wedge \de x^3 \;,
\end{eqnarray*}
the condition in order that an arbitrary $J_{t,\,\ell}$-anti-invariant real $2$-form $\psi \,=\, A\,v_1+ B\, v_2$
is closed is expressed by
\begin{equation}\label{eq:system-T^2}
\left\{
\begin{array}{l}%{lr}
 A_3-B_1\,\alpha \;=\;0% & \qquad\qquad (123)
\\[5pt]
 A_4-B_2 \;=\;0% & (124)
\\[5pt]
 -A_1\,\alpha-B_3 \;=\;0% & (134)
\\[5pt]
 -B_4\,\alpha-A_2\,\alpha-A\,\alpha_2 \;=\;0% & (234)
\end{array}
\right.
\end{equation}
(here and later on, we write, for example, $A_3$ instead of $\frac{\del A}{\del x^3}$).
By solving \eqref{eq:system-T^2}, we obtain
$$ \psi \;=\; \frac{A}{\alpha}\,v_1 + B\, v_2 \qquad \text{ with } \qquad A,B\in\R \;.$$
Therefore, for small enough $t\neq0$, according to \cite{draghici-li-zhang1}, we have the upper-semicontinuity property
$$ \dim_\R H^{(2,0),(0,2)}_{J_{t,\,\ell}}\left(\T^2_\C\right)_\R \;\leq\; 2 \;=\; \dim_\R H^{(2,0),(0,2)}_J
\left(\T^2_\C\right)_\R\;, $$
from which we get the lower one
$$
\dim_\R H^{(1,1)}_{J_{t,\,\ell}}\left(\T^2_\C\right)_\R \;\geq\; 4 \;=\; \dim_\R H^{(1,1)}_J\left(\T^2_\C
\right)_\R\;.
$$

Now, we turn our attention to the case of dimension greater than $4$. We construct
curves through the standard K\"{a}hler structure on the complex $3$-torus, $\left(\T^3_\C,J_0,\omega_0\right)$. Let
$$ L\;=\;\left(
\begin{array}{ccc|ccc}
 \phantom{+}\ell &  &  &  &  &  \\
 & \phantom{+}0 &  &  &  &  \\
 &  & \phantom{+}0 &  &  &  \\
\hline
  & & & -\ell & & \\
 &  &  &  & \phantom{+}0 & \\
 &  &  &  &  & \phantom{+}0
\end{array}
\right)
\;,$$
where $\ell\in\mathcal{C}^\infty(\R^6;\,\R)$ is a $\Z^6$-periodic function. As before,
defining $J_{t,\,\ell}$ (for small $t$) as in \eqref{eq:Jt-L}, we get the $\omega_0$-compatible
almost-complex structure
$$ J_{t,\,\ell}\;=\;
\left(
\begin{array}{ccc|ccc}
 & & & -\frac{1\,-\,t\,\ell}{1\,+\,t\,\ell} & & \\
 &  &  &  & -1 & \\
 &  &  &  &  & -1 \\
\hline
 \phantom{+}\frac{1\,+\,t\,\ell}{1\,-\,t\,\ell} &  &  &  &  & \\
 & \phantom{+}1 &  &  &  & \\
 &  & \phantom{+}1 &  &  &
\end{array}
\right) \;.
$$
As before, setting
$$
\alpha\;:=:\;\alpha(t,\ell)\;\definizione\; \frac{1\,-\,t\,\ell}{1\,+\,t\,\ell} \;,
$$
it follows that a coframe for the holomorphic cotangent bundle is given by
$$
\left\{
\begin{array}{l}
 \varphi^1_t \;:=\; \de x^1\,+\,\im\,\alpha\,\de x^4 \\[5pt]
 \varphi^2_t \;:=\; \de x^2\,+\,\im\,\de x^5 \\[5pt]
 \varphi^3_t \;:=\; \de x^3\,+\,\im\,\de x^6
\end{array}
\right. \;;
$$
therefore
$$
\left\{
\begin{array}{l}
 \de\varphi^1_t \;=\; \im\,\de\alpha\,\wedge\,\de x^4 \\[5pt]
 \de\varphi^2_t \;=\; 0 \\[5pt]
 \de\varphi^3_t \;=\; 0
\end{array}
\right. \;;
$$
note that $\ell\,=\,\ell\left(x^1,x^4\right)$ gives rise to integrable almost-complex structures, in
fact to K\"{a}hler structures; therefore, in such a case $J_{t,\,\ell}$ is \Cpf. Again, we assume
$\ell\,=\,\ell\left(x^3\right)$ not constant. The $J_{t,\,\ell}$-anti-invariant real closed $2$-forms are
$$ \psi\;=\; \frac{C}{\alpha}\left(\de x^{13}-\alpha\,\de x^{46}\right)\,+\,D\,\left(\de x^{16} - \alpha\, \de x^{34}\right)\,+\,E\, \left(\de x^{23} -
\de x^{56}\right)\,+\,F\,\left(\de x^{26} - \de x^{35}\right) $$
where $C,\;D,\;E,\;F\in\R$.\\
So, for small $t\neq0$, we have the upper-semicontinuity property
$$ \dim_\R H^{(2,0),(0,2)}_{J_{t,\,\ell}}\left(\T^3_\C\right)_\R \;\leq\; 4 \;<\; 6 \;=\;
\dim_\R H^{(2,0),(0,2)}_J\left(\T^3_\C\right)_\R\;. $$
Unfortunately, the explicit computation of $H^{(1,1)}_{J_{t,\,\ell}}\left(\T^3_\C\right)_\R$
seems to be not so simple. In particular, it is not clear if $J_{t,\,\ell}$
remains \Cf\ (note that $J_{t,\,\ell}$ is \Cp\ by \cite[Proposition 2.7]{draghici-li-zhang},
see also \cite[Proposition 3.2]{fino-tomassini}).\vskip12pt
Now, we recall how to construct curves of almost-complex structure via a $J$-anti-invariant
real $2$-form, as in \cite{lee}.\\
Let $\left(M,J\right)$ be an almost-complex manifold; take $g$ a $J$-Hermitian metric and $\gamma$ a
real $2$-form in $\wedge^{2,0}M\,\oplus\, \wedge^{0,2}M$. Define $V$ to be the endomorphism of the
tangent bundle such that
\begin{equation}\label{eq:representation}
\gamma\left(\cdot,\cdot\right) \;=:\; g\left(V\cdot,\cdot\right) \;;
\end{equation}
a direct computation shows that $V\,J\,+\,J\,V\,=\,0$. Therefore, setting
$$ L \;\definizione\; \frac{1}{2}\,V\,J\;,$$
we have $L\,J\,+\,J\,L\,=\,0$. At this point, for small $t$, define $J_{t,\,\gamma}$ as in \eqref{eq:Jt-L}:
$$
J_{t,\,\gamma} \;\definizione\; \left(\id\,-\,t\,L\right)\,J\,\left(\id\,-\,t\,L\right)^{-1} \;;
$$
therefore, $\left\{J_{t,\,\gamma}\right\}_{t\in\left(-\varepsilon,\varepsilon\right)}$ is a curve of
almost-complex structures naturally associated with $\gamma$.\\
We give an example of a \Cpf\ structure on a non-K\"{a}hler manifold such that the stability
property holds along a curve constructed in such a way. Let
$$
\mathrm{Sol}(3) \;\definizione\;
\left\{
\left(
\begin{array}{cccc}
 \mathrm{e}^{cz} & 0 & 0 & x \\
 0 & \mathrm{e}^{-cz} & 0 & y \\
 0 & 0 & 1 & z \\
 0 & 0 & 0 & 1
\end{array}
\right)\in \GL(4;\,\R)
\,\,\,\vert\,\,\, x,y,z\in\R
\right\} \;.
$$
Then $\mathrm{Sol}(3)$ is a completely solvable Lie group. For a suitable $c\in\R$, there exists a cocompact
discrete subgroup $\Gamma(c)$ such that
$$ M \;\definizione\; \Gamma(c) \left\backslash \mathrm{Sol}(3) \right. $$
is a compact $3$-dimensional solvmanifold. Define
$$ N^6(c) \;\definizione\; M \,\times\, M \;.$$
The manifold $N^6(c)$ has been studied in \cite{benson-gordon} as an example of a
cohomologically K\"{a}hler manifold; M. Fern\'{a}ndez, V. Mu\~{n}oz and J. A. Santisteban proved in
\cite{fernandez-munoz-santisteban} that it has no K\"{a}hler structures, although it is formal and it has a
symplectic structure satisfying \eqref{eq:hlc}. In \cite{fino-tomassini} a family of \Cpf\ structures on $N^6(c)$
is provided. Now, we will construct a curve of \Cpf\ almost-complex structures on $N^6(c)$.\\
Let $\left\{e^i\right\}_{i\in\{1,\ldots,6\}}$ be a coframe for $N^6(c)$; the structure equations are
$$
\left\{
\begin{array}{l}
 \de e^1 \;=\; \phantom{+}c\, e^1\wedge e^3 \\[5pt]
 \de e^2 \;=\; -c\, e^2\wedge e^3\\[5pt]
 \de e^3 \;=\; \phantom{+}0\\[5pt]
 \de e^4 \;=\; \phantom{+}c\, e^4\wedge e^6 \\[5pt]
 \de e^5 \;=\; -c\, e^5\wedge e^6 \\[5pt]
 \de e^6 \;=\; \phantom{+}0
\end{array}
\right. \;.
$$
Let $J$ be the almost-complex structure given by
$$
J\;=\;
\left(
\begin{array}{cc|cc|cc}
 & -1 & & & & \\
 \phantom{+}1 & & & & & \\
\hline
 & & & -1& & \\
 & & \phantom{+}1& & & \\
\hline
 & & & & &-1 \\
 & & & &\phantom{+}1 &
\end{array}
\right) \;.
$$
By Hattori-Nomizu theorem one computes immediately
$$ H^2_{dR}\left(N^6(c);\R\right) \;=\; \Span_\R\left\{e^{1}\wedge e^{2},\, e^{3}\wedge e^6-e^{4}\wedge e^{5},\,
e^{3}\wedge e^{6}+e^{4}\wedge e^{5}\right\} \;;$$
hence $N^6(c)$ is a \Cpf\ and \pf\ manifold, the harmonic representatives being of pure degree. Note that
$$
H^{(2,0),(0,2)}_J\left(N^6(c)\right)_\R \;=\; \Span_\R\left\{e^{3}\wedge e^{6}+e^{4}\wedge e^{5}\right\} \;.
$$
Put $\gamma\,:=\,e^{3}\wedge e^{6}+e^{4}\wedge e^{5}$; then the linear map $V$ representing $\gamma$ as in \eqref{eq:representation} is
$$
V\;=\;
\left(
\begin{array}{cc|cc|cc}
 \phantom{+}0& & & & & \\
 &\phantom{+}0 & & & & \\
\hline
 & & & & & -1 \\
 & & & & -1 & \\
\hline
 & & & \phantom{+}1 & & \\
 & & \phantom{+}1 & & &
\end{array}
\right) \;,
$$
and then it is immediate to compute
$$
2L\;=\;
\left(
\begin{array}{cc|cc|cc}
 \phantom{+}0& & & & & \\
 &\phantom{+}0 & & & & \\
\hline
 & & & & -1 & \\
 & & & & & \phantom{+}1 \\
\hline
 & & \phantom{+}1 & & & \\
 & & & -1 & &
\end{array}
\right)
$$
and
$$
J_t \;:=\; J_{t,\,\Phi} \;=\;
\left(
\begin{array}{cc|cc|cc}
 & -1 & & & & \\
 \phantom{+}1 & & & & & \\
\hline
 & & & -\frac{4-t^2}{4+t^2} & & -\frac{4t}{4+t^2} \\
 & & \phantom{+}\frac{4-t^2}{4+t^2} & & -\frac{4t}{4+t^2} & \\
\hline
 & & & \frac{4t}{4+t^2} & &-\frac{4-t^2}{4+t^2} \\
 & & \frac{4t}{4+t^2} & &\frac{4-t^2}{4+t^2} &
\end{array}
\right) \;.
$$
Set
$$ \alpha \;:=:\; \alpha(t) \;\definizione\; \frac{4-t^2}{4+t^2}\;,\qquad \beta\;:=:\;\beta(t) \;\definizione\;
\frac{4t}{4+t^2}\;.$$
A coframe for the $J_t$-holomorphic cotangent bundle is given by
$$
\left\{
\begin{array}{l}
 \varphi^1_t \;=\; e^1+\im e^2 \\[5pt]
 \varphi^2_t \;=\; e^3+\im \left(\alpha\, e^4+\beta\, e^6\right) \\[5pt]
 \varphi^3_t \;=\; e^5+\im \left(-\beta\, e^4+\alpha\, e^6\right)
\end{array}
\right. \;.
$$
The closed $2$-forms
$$ \frac{1}{2\im}\,\varphi_t^{1\bar{1}}\;,\qquad \frac{1}{2\im}\varphi_t^{3\bar{3}}-\frac{\alpha}{c}\,\de e^{5}\;,
\qquad \frac{1}{2\im}\left(\beta\,\varphi_t^{2\bar{2}}+\alpha\left(\varphi_t^{2\bar{3}}-\varphi_t^{\bar{2}3}\right)
\right)+\frac{1}{2\im}\varphi_t^{3\bar{3}} $$
generates three different cohomology classes; hence, for small $t\neq0$, we get
$$
H^{2}_{dR}\left(N^6(c);\R\right) \;=\; H^{(1,1)}_{J_t}\left(N^6(c)\right)_\R \;;
$$
this implies that $J$ is \Cf\ as well as \p. A straightforward computation yields
\begin{eqnarray*}
H^4_{dR}\left(N^6(c);\R\right) &=& \Span_\R
\Bigg\{*_g\left(\frac{1}{2\im}\,\varphi_t^{1\bar{1}}\right), \; *_g\left(\varphi_t^{3\bar{3}}-\frac{\alpha}{c}\,
\de e^{5}\right)+\frac{\alpha}{c}\,\de \left(e^{125}\right),\\[5pt] &&\frac{\alpha}{4}\left(\varphi_t^{12\bar{1}\bar{3}}+
\varphi_t^{\bar{1}\bar{2}13}\right)+\frac{\beta}{4}\,\varphi_t^{12\bar{1}\bar{2}}+\frac{\alpha\,\beta}{c}\,
\de\left(e^{125}\right)\Bigg\} \;=\\[5pt]
&=& H^{(2,2)}_{J_t}\left(N^6(c)\right)_\R \;.
\end{eqnarray*}
Therefore $N^6(c)$ is also \Cf\ at the $4$th stage and, consequently, it is \f\ as well as \Cp.

Summarizing, we have proved the following.
\begin{thm}\label{thm:n6(c)}
 There exists a compact manifold $N^6(c)$ such that:
\begin{enumerate}
\item[(i)] $N^6(c)$ admits a \Cpf\ almost-complex structure $J$;
\item[(ii)] each harmonic form of type $(2,0)+(0,2)$ gives rise to a curve
$\left\{J_t\right\}_{t\in\left(-\varepsilon,\varepsilon\right)}$ of \Cpf\ almost-complex structures on $N^6(c)$;
\item[(iii)] furthermore, the map
$$
t\mapsto \dim_\R H^{(2,0),(0,2)}_{J_t}\left(N^6(c)\right)_\R
$$
is an upper-semicontinuous function at $t=0$.
\end{enumerate}
\end{thm}
\bigskip

\noindent{\sl Acknowledgments.} We would like to thank Tedi Dr\v{a}ghici, Tian-Jun Li and Weiyi Zhang for
their very useful comments and for pointing us the reference \cite{draghici-li-zhang1}. \newline
We are also pleased to thank the referee for fruitful suggestions and remarks for a
better presentation of the results.

\bigskip\noindent

\begin{thebibliography}{10}
\bibitem{alessandrini-andreatta}
L. Alessandrini, M. Andreatta, Closed transverse {$(p,p)$}-forms
  on compact complex manifolds, {\em Compositio Math.} \textbf{61} (1987), no.~2,
  181--200.

\bibitem{alessandrini-bassanelli}
L. Alessandrini, G. Bassanelli, Small deformations of a class
  of compact non-K\"ahler manifolds, {\em Proc. Amer. Math. Soc.} \textbf{109}
  (1990), no.~4, 1059--1062.

\bibitem{audin-lafontaine}
M. Audin, J. Lafontaine (eds.), \emph{Holomorphic curves in
  symplectic geometry}, Progress in Mathematics, vol. 117, Birkh\"auser Verlag,
  Basel, 1994.

\bibitem{barth-peters-vandeven}
W.~P. Barth, K. Hulek, C. A.~M. Peters, A. Van~de Ven,
  \emph{Compact complex surfaces}, second ed., Ergebnisse der Mathematik und
  ihrer Grenzgebiete. 3. Folge. A Series of Modern Surveys in Mathematics
  [Results in Mathematics and Related Areas. 3rd Series. A Series of Modern
  Surveys in Mathematics], vol.~4, Springer-Verlag, Berlin, 2004.

\bibitem{B} F.~A. Belgun, On the metric structure of
non-K\"ahler complex surfaces, \emph{ Math. Ann.} {\bf 317} (2000), 1--40.

\bibitem{benson-gordon1}
C. Benson, C.~S. Gordon, K\"ahler and symplectic structures on nilmanifolds,
\emph{Topology} \textbf{27} (1988), 513--518.

\bibitem{benson-gordon}
C. Benson,  C.~S. Gordon, K\"ahler structures on compact
  solvmanifolds, {\em Proc. Amer. Math. Soc.} \textbf{108} (1990), no.~4, 971--980.

\bibitem{console-fino}
S. Console, A. Fino, Dolbeault cohomology of compact
  nilmanifolds, {\em Transform. Groups} \textbf{6} (2001), no.~2, 111--124.

\bibitem{derham}
G. de~Rham, \emph{Differentiable manifolds}, Grundlehren der
  Mathematischen Wissenschaften [Fundamental Principles of Mathematical
  Sciences], vol. 266, Springer-Verlag, Berlin, 1984, Forms, currents, harmonic
  forms, Translated from the French by F. R. Smith, With an introduction by S.
  S. Chern.

\bibitem{deligne}
P. Deligne, Ph. Griffiths, J. Morgan, D.  Sullivan, Real
  homotopy theory of K\"ahler manifolds, {\em Invent. Math.} \textbf{29} (1975),
  no.~3, 245--274.

\bibitem{demailly-agbook}
J.-P. Demailly, \emph{{C}omplex {A}nalytic and {D}ifferential
  {G}eometry},
  \url{http://www-fourier.ujf-grenoble.fr/~demailly/manuscripts/agbook.pdf}, 2012.

\bibitem{donaldson}
S.~K. Donaldson, \emph{Two-forms on four-manifolds and elliptic equations},
  Inspired by {S}. {S}. {C}hern, Nankai Tracts Math., vol.~11, World Sci.
  Publ., Hackensack, NJ, 2006, pp.~153--172.

\bibitem{draghici-li-zhang}
T. Dr\v{a}ghici, T.-J. Li, W.  Zhang, Symplectic forms and cohomology
decomposition of almost complex four-manifolds, {\em Int. Math.
Res. Not. IMRN} {\bf 2010} (2010), no.~1, 1--17.

\bibitem{draghici-li-zhang1}
T. Dr\v{a}ghici, T.-J. Li, W.  Zhang, On the $J$-anti-invariant cohomology of almost complex $4$-manifolds,
\texttt{arXiv:1104.2511v1 [math.SG]}, \textsc{doi:} \texttt{10.1093/qmath/har034}, to appear in {\em Q. J. Math.}.


\bibitem{fernandez-gray}
M. Fern{\'a}ndez, A. Gray, \emph{The {I}wasawa manifold},
  Differential geometry, {P}e\~n\'\i scola 1985, Lecture Notes in Math.,
  vol.~{\bfseries 1209}, Springer, Berlin, 1986, ~157--159.

\bibitem{fernandez-deleon-saralegui} M. Fern\'andez, M. de Le\'on, M. Saralegui,
A six dimensional Compact Symplectic Solvmanifold without
K\"ahler Structures, {\em Osaka J. Math} {\bf 33} (1996), no.~1, 19--34.

\bibitem{fernandez-munoz-santisteban}
M. Fern{\'a}ndez, V. Mu{\~n}oz, J.~A. Santisteban,
  Cohomologically {K}\"ahler manifolds with no {K}\"ahler metrics, {\em Int.
  J. Math. Math. Sci.} (2003), no.~52, 3315--3325.

\bibitem{FG}
A. Fino, G. Grantcharov,
On some properties of the Manifolds with Skew-Symmetric Torsion,
{\em Adv. Math.} {\bf 189} (2004), 439--450.

\bibitem{FPS}
A. Fino, M. Parton, S. M.  Salamon,
Families of strong KT structures in six dimensions,
{\em Comment. Math. Helv.} {\bf 79} (2004), 317--340.

\bibitem{fino-tomassini}
A. Fino, A. Tomassini, On some cohomological properties of
  almost complex manifolds, {\em J. of Geom. Anal.} \textbf{20} (2010), 107--131.

\bibitem{hattori}
A. Hattori, Spectral sequence in the de {R}ham cohomology of fibre
  bundles, {\em J. Fac. Sci. Univ. Tokyo Sect. I} \textbf{8} (1960), 289--331
  (1960).

\bibitem{kodaira-spencer-3}
K. Kodaira, D.~C. Spencer, On deformations of complex
  analytic structures. III. Stability theorems for complex structures,
  {\em Ann. of Math.} (2) \textbf{71} (1960), 43--76.

\bibitem{lee}
J. Lee, Family Gromov-Witten invariants for K\"ahler surfaces,
  {\em Duke Math. J.} \textbf{123} (2004), no.~1, 209--233.

\bibitem{li-zhang}
T.-J. Li, W. Zhang, Comparing tamed and compatible symplectic
  cones and cohomological properties of almost complex manifolds, {\em Comm. Anal. Geom.} \textbf{17} (2009), no. 4,  651--684.
%  {\tt arXiv:0708.2520v4 [math.SG]}.
%\texttt {http://www.math.umn.edu/~zhang393/docs/tame.pdf}

\bibitem{M}
J. Milnor,
Curvature of left-invariant metrics on Lie groups, {\em Adv. in Math.} {\bf 21} (1976), 293--329.

\bibitem{nakamura}
I. Nakamura, Complex parallelisable manifolds and their small
  deformations, {\em J. Differential Geometry} \textbf{10} (1975), 85--112.

\bibitem{nomizu}
K. Nomizu, On the cohomology of compact homogeneous spaces of
  nilpotent Lie groups, {\em Ann. of Math.} (2) \textbf{59} (1954), 531--538.

\bibitem{sakane}
Y. Sakane, On compact complex parallelisable solvmanifolds, {\em Osaka J.
  Math.} \textbf{13} (1976), no.~1, 187--212.

\bibitem{ST}
J. Streets, G. Tian, A parabolic flow of pluriclosed metrics, {\tt arXiv:0903.4418 [math.DG]}, \textsc{doi:} \texttt{10.1093/imrn/rnp237},
{\em Int. Math. Res. Not.} \textbf{2010} (2010), no.~16, 3101--3133.
\end{thebibliography}
\end{document}